\documentclass[11pt]{amsart}

\usepackage{amsfonts, amstext, amsmath, amsthm, amscd, amssymb}
\usepackage{epsfig, graphics, psfrag}
\usepackage{color}

\renewcommand{\setminus}{{\smallsetminus}}

\newcommand{\HH}{{\mathbb{H}}}
\newcommand{\RR}{{\mathbb{R}}}
\newcommand{\ZZ}{{\mathbb{Z}}}

\newcommand{\QQ}{{\mathbb{Q}}}

\newcommand{\vol}{{\rm vol}}
\newcommand{\area}{{\rm area}}
\newcommand{\lam}{{\lambda_1}}

\newcommand{\abs}[1]{{\left\vert #1 \right\vert}}
\newcommand{\lmin}{{\ell_{\rm min}}}

\def\dfr{double coil }
\theoremstyle{plain}
\newtheorem{theorem}{Theorem}[section]
\newtheorem{corollary}[theorem]{Corollary}
\newtheorem{lemma}[theorem]{Lemma}
\newtheorem{prop}[theorem]{Proposition}

\newtheorem*{namedtheorem}{\theoremname}
\newcommand{\theoremname}{testing}
\newenvironment{named}[1]{\renewcommand{\theoremname}{#1}\begin{namedtheorem}}{\end{namedtheorem}}

\theoremstyle{definition}
\newtheorem{define}[theorem]{Definition}
\newtheorem{remark}[theorem]{Remark}

\begin{document}
\title[On diagrammatic bounds of knot volumes and spectral
invariants]{On diagrammatic bounds of knot volumes \\ and spectral
invariants}
\author[D. Futer]{David Futer}
\author[E. Kalfagianni]{Efstratia Kalfagianni}
\author[J. Purcell]{Jessica S. Purcell}

\address[]{Department of Mathematics, Temple University,
Philadelphia, PA 19122}

\email[]{dfuter@math.temple.edu}

\address[]{Department of Mathematics, Michigan State University, East
Lansing, MI, 48824}

\email[]{kalfagia@math.msu.edu}

\address[]{ Department of Mathematics, Brigham Young University,
Provo, UT 84602}

\email[]{jpurcell@math.byu.edu }

\thanks{{E.K. is supported in part by NSF grant DMS--0306995 and
NSF--FRG grant DMS-0805942.}}

\thanks{{J.P. is supported in part by NSF grant DMS-0704359.}}

\thanks{ \today}

\begin{abstract}  
In recent years, several families of hyperbolic knots have been shown
to have both volume and $\lam$ (first eigenvalue of the Laplacian) 
bounded in terms of the twist number of
a diagram, while other families of knots have volume bounded by a
generalized twist number.  We show that for general knots, neither the
twist number nor the generalized twist number of a diagram can provide
two--sided bounds on either the volume or $\lam$.  We do so by
studying the geometry of a family of hyperbolic knots that we call
\emph{\dfr knots}, and finding two--sided bounds in terms of the knot
diagrams on both the volume and on $\lam$.  We also extend a result of
Lackenby to show that a collection of \dfr knot complements forms an
\emph{expanding family} iff their volume is bounded.
\end{abstract}

\maketitle

\section{Introduction}
For any diagram of a knot, there is an associated 3--manifold: the
complement of the knot in the 3--sphere.  In the 1980's, Thurston
proved that the complement of any non-torus, non-satellite knot admits
a hyperbolic metric \cite{bonahon:geometries}, which is necessarily
unique up to isometry.  As a result, geometric information about a
knot complement, such as volume and the spectrum of the Laplacian,
gives topological knot invariants.  However, in practice, these
invariants have been difficult to estimate with only a diagram of a
knot.

Recently, there has been some progress in estimating geometric
information from particular classes of diagrams.  For volumes,
Lackenby showed that the volume of an alternating knot complement is bounded
above and below in terms of the twist number of an alternating diagram
\cite{lackenby:volume} (see Definition \ref{def:twist-num} below).  We
extended these results to highly twisted knots \cite{fkp:volume} and to sums
of alternating tangles \cite{fkp:tangles}. Purcell used a generalization of the twist number to find volume  lower
bounds for additional classes of knots
\cite{purcell:aug-ref-tw}, while in  
\cite{fkp:farey} we showed that the volume of a
closed 3--braid is bounded above and below in terms of
the generalized twist number of the braid.  More recently, Lackenby showed that for
alternating and highly twisted knots, the first eigenvalue $\lam$ of
the Laplacian can be estimated in terms of the twist number
\cite{lackenby:spectrum}.  Based on these examples, one might hope
that a suitable generalization of the twist number of a diagram
controls the geometry of all hyperbolic knot complements.

In this paper, we show that the twisting in a diagram cannot give
two--sided geometric bounds for general knots.  We do so by presenting
a class of knots, called \emph{\dfr knots}, for which the volume can
be made bounded while the twist number becomes arbitrarily large, or
the volume can be made unbounded while the generalized twist number
stays constant.  
Similarly, we show 
that $\lam$ can stay
bounded while the twist number becomes arbitrarily large, or $\lam$
can approach $0$ while the generalized twist number stays constant.


To state our results more precisely, we need a few definitions.

\begin{define}\label{def:twist-num}
	A diagram of a knot is a $4$--valent graph with over--under crossing
	information at each vertex.  A \emph{twist region} of a diagram is a
	portion of the diagram consisting of bigons arranged end to end,
	which is maximal in the sense that there are no additional bigons
	adjacent to either end.  A single crossing not adjacent to any
	bigons is also a twist region.  See Figure \ref{fig:twists}, left. Note
	that a twist region containing $c$ crossings corresponds to $c/2$
	full twists of the two strands.

	The number of twist regions in a particular diagram $D$ is called
	the \emph{twist number of $D$}.  The minimum of the twist numbers of
	$D$ as $D$ ranges over all diagrams of a knot $K$ is defined to be
	the \emph{twist number} of $K$, and is denoted $\tau(K)$.
\end{define}
	
\begin{figure}[h]
\includegraphics{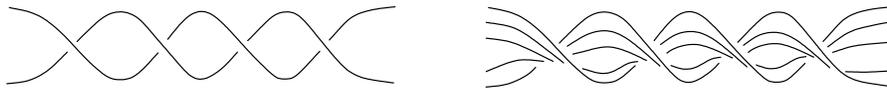}
\caption{Left: a twist region.  Two strands twist about each other
	maximally.  Right: a generalized twist region with two full twists.}
\label{fig:twists}
\end{figure}

\begin{define}\label{def:gen-twist-num}
	A \emph{generalized twist region} on $q$ strands, $q\geq 2$, is a
	region of a knot diagram consisting of $q$ strands twisted
	maximally.  That is, if the $q-2$ innermost strands are removed from
	a generalized twist region on $q$ strands, then the remaining two
	strands form a twist region as in Definition \ref{def:twist-num}.
	These two outermost strands bound a twisted, rectangular ribbon.
	The additional $q-2$ strands are required to run parallel to the two
	outermost strands, embedded on this ribbon. By definition,
	a twist region is a generalized twist on $q=2$ strands. See Figure
	\ref{fig:twists}, right.
	
	In a given diagram $D$, there is typically more than one way to 
	partition the crossings of $D$ into generalized twist regions. For 
	example, a single generalized twist region can contain many ordinary
	twist regions. The \emph{generalized twist number of $D$} is defined 
	to be the smallest number of generalized twist regions, minimized over
	all partitions of $D$ into generalized twist regions.
\end{define}

\begin{figure}[h]
\begin{center}
\vspace{-0.2in}
\includegraphics{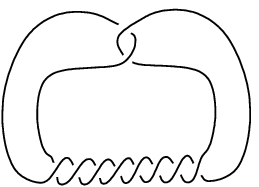}
\hspace{.5in}
\includegraphics{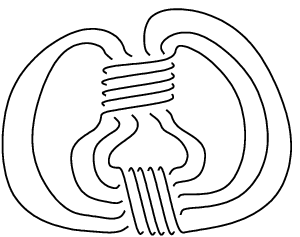}
\end{center}
\caption{
Left: a $(1,2)$ \dfr knot. Right: a $(3,5)$ \dfr knot.
}
\label{fig:dfr-knot}
\end{figure}

\begin{define}\label{def:dfr-knot}
A \emph{\dfr knot} is a knot with exactly two generalized twist
regions, where each twist region contains $q \geq 2$ strands and an
integral number of full twists.  At each end of each generalized twist region,
$p<q$ strands split off to the right, while $q-p$ strands split off to the left. 
A knot $K$ with this description is called a \emph{$(p,q)$ \dfr knot}. Note that 
$K$ will be a knot precisely when $p$ and $q$ are relatively prime.

The integers $p$ and $q$, together with the number of full twists in each generalized twist region, completely specify a diagram of a \dfr knot. See Figure \ref{fig:dfr-knot} for two examples.
\end{define}

Note that when $q=2$ and one of the two generalized twist regions
contains exactly one full twist, then corresponding \dfr knot is a
twist knot. See Figure \ref{fig:dfr-knot}, left.  Thus \dfr knots
can be seen as a generalization of twist knots.

Every \dfr knot is also a special case of a \emph{double torus knot}:
it can be embedded on an unknotted genus--2 surface in $S^3$.  To
visualize this genus--2 surface, start with the sphere obtained by
compactifying the projection plane, and add one handle for each
generalized twist region.  Then, in each region of Figure
\ref{fig:dfr-knot}, the coils run around the cylinder of the handle.
The family of double torus knots has been studied extensively (see
e.g. \cite{hill:double-torus, hill-murasugi:double-torus,
hirasawa-murasugi:survey}).


In Section \ref{sec:augmented}, we prove the following two--sided, combinatorial estimate on the volumes of double coil knots.

\begin{named}{Theorem  \ref{thm:simple-volume}}
Let $p$ and $q$ be relatively prime integers with $0<p<q$, and let $k$ be the length of the continued fraction expansion of $p/q$. Let $K$ be a $(p,q)$ double coil knot, in which each generalized twist region has at least $4$ full twists. Then $K$ is hyperbolic, and
$$0.9718 \, k - 0.3241 \: \leq \:  \vol(S^3 \setminus K) \: < \: 4v_8k,$$
where $v_8 = 3.6638...$ is the volume of a regular ideal octahedron in
$\HH^3$.
\end{named}

The length of the continued fraction expansion of $p/q$ turns out to
be unrelated to either the twist number or the generalized twist
number of a $(p,q)$ \dfr knot. As a result, we can show that neither
of those quantities predicts the volume of $K$.

\begin{named}{Theorem \ref{thm:no-easy-estimate}}
The volumes of hyperbolic \dfr knots are not effectively predicted by
either the twist number or generalized twist number.  More precisely:
\begin{enumerate}
\item[(a)] For any $q \geq 3$, and any $p$ relatively prime to $q$,
there exists a sequence $K_n$ of $(p,q)$ \dfr knots such that
$\tau(K_n) \to \infty$ while $\vol(S^3 \setminus K_n)$ stays bounded.
\item[(b)] All double coils have generalized twist number $2$, but their
volumes are unbounded.
\end{enumerate}
\end{named}

Theorem \ref{thm:no-easy-estimate} implies that the known upper bounds
on volume in terms of twist number can be quite ineffective.  Lackenby
initially found an upper bound on volume that was linear in terms of
twist number \cite{lackenby:volume}.  Agol and D. Thurston improved
the constants in Lackenby's estimate, and showed that the upper bound
is asymptotically sharp for a particular family of alternating links
\cite[Appendix]{lackenby:volume}.  However, for the \dfr knots of
Theorem \ref{thm:no-easy-estimate}, the volume is bounded but the
estimate in terms of twist number will become arbitrarily large. This
phenomenon occurs in much greater generality: see Theorem
\ref{thm:big-twisting} for the most general statement, and Corollary
\ref{cor:braids} for an application to $m$--braids.


For a Riemannian manifold $M$, $\lam(M)$ is defined to be the smallest
positive eigenvalue of the Laplace--Beltrami operator $\Delta f = -
{\rm div \, grad} f$. It turns out that the volume and $\lam$ of a
\dfr knot are closely related. In Section \ref{sec:laplacian}, we show
the following result.

\begin{named}{Theorem \ref{thm:dfr-lambda-bound}}
Let $K$ be a hyperbolic \dfr knot. Then
$$ \frac{A_1}{\vol(S^3 \setminus K)^2} \: \leq \: \lam(S^3 \setminus
K) \: \leq \: \frac{A_2}{\vol(S^3 \setminus K)},$$
where $A_1 \geq  8.76 \times 10^{-15}$ and $A_2 \leq 12650$.
\end{named}

Combining Theorem \ref{thm:no-easy-estimate} with Theorem
\ref{thm:dfr-lambda-bound} immediately gives the following.

\begin{named}{Corollary \ref{cor:lambda-twist}}
The first eigenvalue of the Laplacian of hyperbolic \dfr knots is not
effectively predicted by either the twist number or generalized twist
number.
More precisely:
\begin{enumerate}
\item[(a)] For any $q \geq 3$, and any $p$ relatively prime to $q$,
there exists a sequence $K_n$ of $(p,q)$ double coils such that
$\tau(K_n) \to \infty$ while $\lam(S^3 \setminus K_n)$ is bounded away
from $0$ and $\infty$.
\item[(b)] All \dfr knots have generalized twist number $2$, but the
	infimum of $\{\lam(S^3 \setminus K_n)\}$ is zero.
\end{enumerate}
\end{named}


Theorem \ref{thm:dfr-lambda-bound} also extends a result of
Lackenby about expanding families.  Recall that a collection $\{M_i\}$
of Riemannian manifolds is called an \emph{expanding family} if $\inf
\lam(M_i)>0$.  Lackenby showed that knots whose volumes are bounded
above and below by the twist number form an expanding family if and
only if their volumes are bounded \cite[Theorem
1.7]{lackenby:spectrum}.  Even though the volumes of \dfr knots are
very far from being governed by the twist number, Theorem 
\ref{thm:dfr-lambda-bound} implies that a sequence of \dfr knots forms 
an expanding family if and only if their volumes are bounded.

\medskip

This paper is organized as follows. In section \ref{sec:augmented}, we
study the geometry and combinatorics of a certain surgery parent of \dfr knots.
The volume estimates for these parent links lead to volume estimates for \dfr
knots in Theorem \ref{thm:simple-volume}. In Section \ref{sec:volume-twist}, we construct hyperbolic knots that have bounded volume but arbitrarily
large twist number. In Section \ref{sec:laplacian}, we describe the connection between
the volume of a \dfr knot and its first eigenvalue $\lam$.  Our main tool here is 
Theorem \ref{thm:lambda-general},
which gives two--sided bounds for $\lam$ of a
finite--volume hyperbolic 3--manifold, in terms of the volume and the
Heegaard genus.

\section{Volume estimates for \dfr knots}\label{sec:augmented}

In this section, we study the volumes of \dfr knots. We begin by showing that a $(p, q)$ \dfr 
knot is obtained by Dehn filling
a certain 3--component link, closely related to the 2--bridge link of slope $p/q$.
Next, we obtain two sided diagrammatic bounds on volumes of these parent links.  
Finally, we apply a result of the authors \cite{fkp:volume} to bound the change in volume 
under Dehn filling, obtaining two-sided diagrammatic
estimates on the volume of the \dfr knots.

\subsection{Augmentations of \dfr knots}

A twist knot as in Figure \ref{fig:dfr-knot}(a) may be viewed as a Dehn filling of the Whitehead link,
which is itself a Dehn filling of the Borromean rings.  Similarly, we
may view \dfr knots as Dehn fillings of a class of link complements in
$S^3$.  The idea is as follows.  At each of the two generalized twist
regions of a \dfr knot, insert a \emph{crossing circle} $C_i$, namely a
simple closed curve encircling all $q$ strands of the generalized
twist.  The complement of the resulting three--component link is
homeomorphic to the complement of the three--component link with all
full twists removed from each twist region.  Examples of such links
are shown in Figure \ref{fig:dfr-link}.  We call such a link the
\emph{augmentation} of a \dfr knot.

\begin{figure}[h]
\begin{center}
	\includegraphics{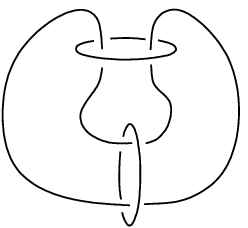}
	\hspace{.5in}
	\includegraphics{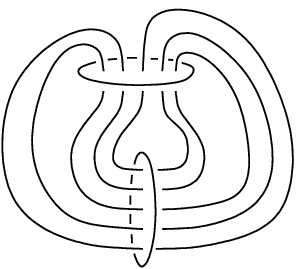}
\end{center}
\caption{Examples of links obtained by adding crossing circles to \dfr
	knots and untwisting.}
\label{fig:dfr-link}
\end{figure}



The augmentation of a $(p,q)$ \dfr knot has a
simple description in terms of the rational number $p/q$, as
follows.  The augmentation consists of three components.  Two, namely
$C_1$ and $C_2$, can be isotoped to lie orthogonal to the projection
plane, bounding simple disjoint disks $D_1$ and $D_2$
in $S^3$.  The third component can be isotoped to be a nontrivial
simple closed curve embedded on the projection plane, disjoint from
the intersections of $C_1$ and $C_2$ with the projection plane. We
adopt the convention that the projection plane contains a point at
infinity, forming a sphere in $S^3$. Note that the projection sphere
minus the four points of intersection with $C_1$ and $C_2$ is a
4--punctured sphere $S$.  Once we have determined a framing for $S$,
any simple closed curve can be described by a number in $ \QQ \cup \{1/0\}$.
		
We choose our framing as follows.  Let $1/0 = \infty$ be the simple
closed curve on $S$ that is disjoint from $D_1$ and $D_2$, and
separates those disks from each other.
Now, draw a straight arc $A$ connecting one of the punctures of $C_1$
with one of $C_2$, as in Figure \ref{fig:gens}(a).  Let the simple
closed curve encircling this arc be $0/1 = 0$.  Note that, in
choosing $A$, there is a $\ZZ$--worth of choices up to isotopy; by
Lemma \ref{lemma:integer-twist}, this ambiguity turns out to be
immaterial.

\begin{figure}
\begin{center}
(a) \input{figures/gen-v2.pstex_t} \hspace{.5in}
(b) \includegraphics{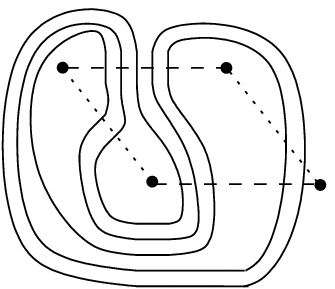}
\end{center}
\caption{(a) A framing for the 4--punctured sphere. (b) The curve
$2/5$.}
\label{fig:gens}
\end{figure}

Given a fixed meaning for $1/0$ and $0/1$, as well as an orientation
on the $4$--punctured projection sphere $S$, every curve on $S$ is
determined by a number $p/q \in \QQ \cup \{1/0\}$, where $p$ and $q$
are relatively prime.  Concretely, this curve can be drawn by marking
$q$ ticks on the arcs corresponding to $D_1$ and $D_2$, and $p$ ticks
on the arcs $A$ and $A'$ of Figure \ref{fig:gens}(a), and then
connecting the dots, as in Figure \ref{fig:gens}(b).

\begin{define}\label{def:aug-link}
The three--component link consisting of $C_1$, $C_2$, and the curve of
slope $p/q$ will be denoted $L_{p/q}$.  Thus Figure \ref{fig:dfr-link}
depicts $L_{1/2}$ and $L_{3/5}$.  Note that for $p,q$ relatively prime and $0<p<q$, $L_{p/q}$ is 
the augmentation of a $(p,q)$ \dfr knot. The $(p,q)$ \dfr knot with $n_1$ full twists in one generalized twist region and $n_2$ full twists in the other generalized twist region can be recovered from $L_{p/q}$ by performing $1/n_i$ Dehn filling on $C_i$.
\end{define}

\begin{lemma}
The link $L_{p/q}$ is isotopic to $L_{k+(p/q)}$, by an isotopy that
preserves the projection plane.
\label{lemma:integer-twist}
\end{lemma}

\begin{proof}
In the projection plane, the curves of slope $p/q$ and $k+(p/q)$ are
related by performing $k$ half--Dehn twists about the closed curve of
slope $1/0$.  Note that this curve of slope $1/0$ is the intersection
between the projection plane and a $2$--sphere $\Sigma$ that separates
$D_1$ from $D_2$.  Thus, because $\Sigma$ is disjoint from $C_1$ and
$C_2$, the Dehn twists about its equator can be realized by an isotopy
in $S^3$ that preserves the projection plane and carries $L_{p/q}$ to
$L_{k+(p/q)}$.
\end{proof}

Thus we may assume $0<p<q$, provided $p/q \notin \{ 0, \infty \}$.
The cases in which $p/q = 0$ or $\infty$ do not lead to hyperbolic
links, and so we will assume $q\geq 2$.

\begin{lemma}\label{lemma:aug-symmetries}
The augmentations of double coil knots have the following symmetries:
\begin{enumerate}
\item[(a)] $S^3 \setminus L_{p/q}$ admits an orientation--reversing
involution, namely reflection in the projection plane.

\item[(b)] $S^3 \setminus L_{p/q}$ admits an orientation--preserving
involution interchanging $C_1$ and $C_2$.

\item[(c)] $S^3 \setminus L_{p/q}$ is homeomorphic to $S^3 \setminus
L_{-p/q}$.
\end{enumerate}
\end{lemma}

\begin{proof}
The involution in (a) is immediately visible in Figure
\ref{fig:dfr-link}.  The involution in (b) is a $\pi$--rotation about
an axis perpendicular to $S$.  Within $S$, the involution is a
$\pi$--rotation about two points (in Figure \ref{fig:gens}, the center
of the parallelogram and the point at infinity), which sends the curve
of slope $p/q$ to an isotopic curve.  Finally, statement (c) is
immediate because $L_{p/q}$ becomes $L_{-p/q}$ when viewed from the
other side of the projection plane.
\end{proof}

\subsection{2--bridge links and augmented 2--bridge links}

The links $L_{p/q}$ are related in a fundamental way to 2--bridge
links.  In order to show this relationship, we present the following
method for constructing links.

Let $S$ denote the 4--punctured sphere.  Consider $S \times [0,1]$
embedded in $S^3$, with the framing on $S = S\times \{t\}$ as above,
for all $t\in [0,1]$.

Recall that we may obtain (the complement of) any 2--bridge link by
attaching two 2--handles to $S\times [0,1]$, one along the slope $1/0$
on $S\times \{1\}$, and one along a slope $p/q$ on $S\times \{0\}$.
Since Dehn twisting along $1/0$ gives a homeomorphic link, we may
assume $p/q \in \QQ/\ZZ$.  The continued fraction expansion of $p/q$
now describes an alternating diagram of the 2--bridge link.  See
\cite[Proposition 12.13]{burde-zieschang}.  One example is depicted in
Figure \ref{fig:2-bridge}(a).

We modify this construction slightly.  Attach a 2--handle to $S \times
\{1\}$ along the slope $1/0$ as before.  However, on $S\times \{0\}$,
chisel out the slope $p/q$.  This separates $S\times \{0\}$ into two
2--punctured disks.  Glue one 2--punctured disk to the other, gluing
the boundary corresponding to the slope $p/q$ to itself, and gluing
the other boundary components in pairs.  We call this link the
\emph{clasped 2--bridge link of slope $p/q$}.  See Figure
\ref{fig:2-bridge}(b) for an example.

(Note that up to homeomorphism, there are two ways to glue the
2--punctured disks so that the boundary $p/q$ is glued to itself.
Either way is acceptable and leads to the same results below: any
extra crossing cancels with its mirror image in Proposition
\ref{thm:2bridge}.)

\begin{figure}[h]
\input{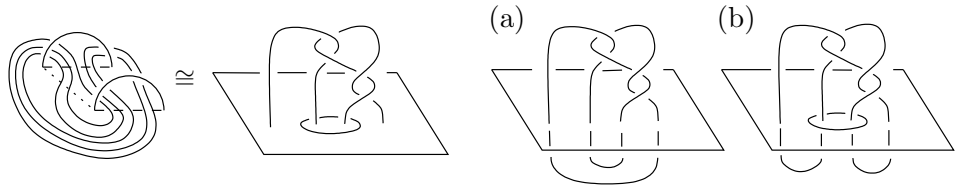}
\caption{(a).  Constructing the 2--bridge knot of slope $2/5$.  (b).
	Constructing the clasped 2--bridge link of slope $2/5$.}
\label{fig:2-bridge}
\end{figure}

\begin{remark}\label{rem:aug-2bridge}
Note that the clasped 2--bridge link of slope $p/q$ has a diagram
similar to the diagram of the regular 2--bridge link of slope $p/q$,
as in Figure \ref{fig:2-bridge}.  In particular, the diagrams will be
identical ``above'' the embedded surface $S \times \{0\}$, and here we
take both diagrams to agree with the standard alternating diagram of
the 2--bridge link.  On $S \times \{0\}$, the clasped 2--bridge link
will have an extra link component, the \emph{clasp} component, which
bounds two embedded 2--punctured disks in $S \times \{0\}$.  Below $S
\times \{0\}$, both diagrams consist of two simple arcs, but they are
attached to differing punctures of $S \times \{0\}$ for the 2--bridge
link and for the clasped 2--bridge link.  Compare the examples in
Figure \ref{fig:2-bridge}(a) and (b).

Note also that by performing $\pm 1/N$ Dehn filling about the clasp
component, we replace the clasp and the two strands it encircles by
$N$ full twists of those two strands (in other words, a twist region
with $2N$ crossings).  By choosing the sign of the Dehn filling
appropriately, we can ensure that the result is the alternating
diagram of a 2--bridge link of some new slope.  Thus the clasped
2--bridge link of slope $p/q$ can be viewed as an augmented 2--bridge
link of some other slope, where we are using the term \emph{augmented}
in the sense of Adams \cite{adams:aug}.
\end{remark}


There is a standard way to add two manifolds containing embedded
2--punctured disks, explored by Adams \cite{adams:3-punct}.  This is
the belted sum.  We recall the definition.

\begin{figure}[h]
\input{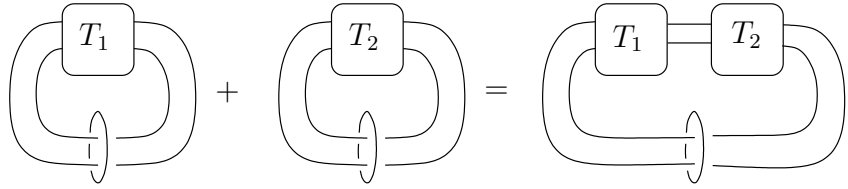}
\caption{A belted sum.}
\label{fig:beltsum}
\end{figure}

Let $M_1$ be the complement of the link in $S^3$ with the following
presentation.  $T_1$ is some tangle in a $4$--punctured sphere.  The
four punctures of this sphere are connected in a manner as shown on
the left in Figure \ref{fig:beltsum}, with a simple closed unknotted
curve $B_1$ encircling the two strands.  Note $B_1$ bounds a
2--punctured disk $S_1$ in the complement of the link.  We will call
the link component $B_1$ the \emph{belt component} of the link.  We
will call the link consisting of $T_1$ and the component $B_1$ a
\emph{belted link}.  We will only be interested in belted links
admitting hyperbolic structures.

Given two hyperbolic belted links with complements $M_1$ and $M_2$,
consisting of tangles $T_1$ and $T_2$, belt components $B_1$ and
$B_2$, and 2--punctured disks $S_1$ and $S_2$, we form the complement
of a new belted link as follows.  Cut each manifold $M_1$ and $M_2$
along $S_1$ and $S_2$, respectively.  We obtain two manifolds,
$\breve{M_1}$ and $\breve{M_2}$, each with two 2--punctured disks as
boundary.  There is a unique hyperbolic structure on a 2--punctured
disk, hence any two are isometric.  This allows us to glue
$\breve{M_1}$ to $\breve{M_2}$ by isometries of their boundaries,
gluing $B_1$ to $B_2$.  The result is the complement of a new belted
link.  See Figure \ref{fig:beltsum}.  We call this new belted link the
\emph{belted sum} of $T_1$ and $T_2$.

\begin{figure}[h]
\input{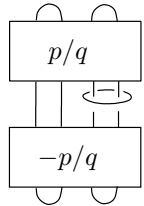}
\caption{The belted sum whose complement is homeomorphic to
$S^3 \setminus L_{p/q}$.}
\label{fig:scheme}
\end{figure}

\begin{prop}
$S^3 \setminus L_{p/q}$ is homeomorphic to the belted sum of a clasped
2--bridge link of slope $p/q$ and a clasped 2--bridge link of slope
$-p/q$. (See Figure \ref{fig:scheme}.)
\label{thm:2bridge}
\end{prop}

\begin{proof}
Slice $S^3 \setminus L_{p/q}$ along the projection plane.  This cuts
the manifold into two pieces, which we call the top half and bottom
half, each bounded by a 4--punctured sphere.  Let $S$ be a
4--punctured sphere.  Note we can embed $S\times[0,1]$ in the top half
in $S^3$ such that $S\times \{0\}$ is embedded on the projection
plane, and the punctures of $S\times \{t\}$ correspond to points on
the crossing circles $C_1$ and $C_2$.

Alternately, give $S$ the same framing as above, so that $1/0$
corresponds to the curve encircling $D_1$ (or $D_2$), and attach a
2--handle to $S\times \{1\}$ along the slope $1/0$.  By our choice of
framing, the result is homeomorphic to capping off the halves of arcs
$C_1$ and $C_2$.  When we chisel out the slope $p/q$ on the projection
plane $S\times \{0\}$, the result is a manifold with boundary
consisting of two 2--punctured disks.  This is homeomorphic to the top
half of $S^3\setminus L_{p/q}$, as in the left of Figure
\ref{fig:2-bridge}.  Thus the top half of $S^3 \setminus L_{p/q}$
sliced along the projection plane is homeomorphic to a clasped
2--bridge link of slope $p/q$ sliced open along the 2--punctured disk
bounded by the clasp component.

Similarly, when we consider the bottom half of $S^3 \setminus L_{p/q}$
sliced along the projection plane, we see it is homeomorphic to the
clasped 2--bridge link of slope $-p/q$, sliced open along the
2--punctured disk bounded by the clasp component.

Since we glue the 2--punctured disks of the top half to those of the
bottom half such that the chiseled--out curve $p/q$ is glued to
itself, this is, by definition, a belted sum of the two manifolds.
\end{proof}

\subsection{Volume bounds for the parent links}

Recall that every rational number $p/q \in \QQ$ can be expressed as a
finite length continued fraction.  When $q>p>0$ and all the terms of
the continued fraction are positive, this expression is unique.  We
define the \emph{length} of the continued fraction to be the number of
denominators in this unique continued fraction where all denominators
are positive.

\begin{theorem}
Let $k$ be the length of the continued fraction expansion of $p/q$,
with $0<p<q$ and $q\geq 2$.  Then $L_{p/q}$ is hyperbolic, and
$$4k v_3 - 1.3536\leq \vol(S^3 \setminus L_{p/q}) \leq 4k v_8,$$ where
$v_3 = 1.0149...$ is the volume of a regular ideal tetrahedron and
$v_8 = 3.6638...$ is the volume of a regular ideal octahedron in
$\HH^3$.
\label{thm:parent}
\end{theorem}

\begin{proof}
By Proposition \ref{thm:2bridge}, $S^3 \setminus L_{p/q}$ is
homeomorphic to the belted sum of two clasped 2--bridge links, one of
slope $p/q$ and one of slope $-p/q$.  By Remark \ref{rem:aug-2bridge},
a portion of the diagram of each clasped 2--bridge link (essentially,
everything away from the clasp) agrees with the alternating diagram of
the regular 2--bridge link of slope $p/q$ or $-p/q$ respectively.  It
is well known that these 2--bridge links will have exactly $k$ twist
regions (see, for example Burde and Zieschang \cite[Proposition
12.13]{burde-zieschang}).  Thus the clasped 2--bridge links will
contain $k$ twist regions as well as a separate clasp component.  By
our restrictions on $q$ and $p$, $k \geq 1$.

Now, again by Remark \ref{rem:aug-2bridge},
the clasped 2--bridge link can be Dehn filled, along slope $\pm 1/N$
on the clasp component, to give a diagram of a new alternating
2--bridge link $K_N$.  The link $K_N$ will have $k+1$ twist regions,
with $2N$ crossings in the new $(k+1)$-st twist region.  As $N$
approaches infinity, the limit in the geometric topology of the $K_N$
will be the original clasped 2--bridge link \cite{thurston:notes}.
Because $k+1 \geq 2$, each link $K_N$ is hyperbolic (see
e.g. \cite[Theorem A.1]{gf:two-bridge}).  Thus its geometric limit is
also hyperbolic.  Finally, the belted sum of hyperbolic manifolds is
hyperbolic.  So $L_{p/q}$ is hyperbolic.

Futer and Gu\'eritaud have found bounds on the volumes of 2--bridge
knots.  By \cite[Theorem B.3]{gf:two-bridge}, the complement of a
2--bridge knot whose standard alternating diagram has $k+1$ twist
regions has volume at least $2(k+1)v_3 -2.7066$ and at most $2kv_8$.
Since the clasped 2--bridge link of slope $\pm p/q$ is the geometric
limit of such manifolds, it satisfies the same volume bounds.  Adams
observed that the volume of a belted sum of two hyperbolic manifolds
is equal to the sum of the volumes of the two pieces
\cite{adams:3-punct}.  Thus the volume of $S^3 \setminus L_{p/q}$ is
at least $4k v_3 -1.3536$ and at most $4v_8k$.
\end{proof}

\subsection{Volume bounds for \dfr knots}

Let $K$ be a $(p,q)$ \dfr knot. Then, by Definition \ref{def:aug-link}, $K$ is
obtained by $1/n_i$ filling on the component $C_i$
of $L_{p/q}$, $i=1, 2$.   We may bound the
volume of $K$ by bounding the change in  volume under Dehn filling.  Our main
tool is the following recent result of the authors, Theorem 1.1 of
\cite{fkp:volume}.


\begin{theorem}[\cite{fkp:volume}]
	Let $M$ be a complete, finite--volume hyperbolic manifold with
	cusps.  Suppose $C_1, \dots, C_k$ are disjoint horoball
	neighborhoods of some subset of the cusps.  Let $s_1, \dots, s_k$ be
	slopes on $\partial C_1, \dots, \partial C_k$, each with length
	greater than $2\pi$.  Denote the minimal slope length by $\lmin$.
	Then the manifold $M(s_1, \dots, s_k)$, obtained by Dehn filling $M$
	along $s_1, \dots, s_k$, is hyperbolic, and
	$$ \vol(M(s_1, \dots, s_k)) \: \geq \:
	\left(1-\left(\frac{2\pi}{\lmin}\right)^2\right)^{3/2} \vol(M).$$
\label{thm:fkp-dehn-filling}
\end{theorem}

We also need the following additional notation.  Let $k$, $n_1$, $n_2$ be integers.  Define
\begin{equation}\label{eq:ell}
 n := \min \left\{ \abs{n_1}, \abs{n_2} \right\} \quad \mbox{and}
	\quad \ell := \max \left\{\frac{1}{4}+4n^2, \: \frac{32\sqrt 2 \,
	k^2 n^2}{7203} \right\}.
\end{equation}
Note that the right--hand term of the maximum becomes greater when $k
\geq 26$.  We may now give volume bounds on \dfr knots.

\begin{theorem}
Let $K$ be a $(p,q)$ \dfr knot, where one generalized twist region contains $n_1$ positive full twists, 
and the other region contains $n_2$ twists.
Let $k$ denote the length of the
continued fraction expansion of $p/q$, and let $\ell$ be as in (\ref{eq:ell}) above.
Suppose that at least one of the following holds:
\begin{enumerate}
	\item $\abs{n_i} \geq 4$ for $i=1, 2$.
	\item $k \, \abs{n_i} \geq 80$ for $i=1, 2$.
\end{enumerate}
Then $K$ is hyperbolic with volume
$$\left(1 - \frac{4\pi^2}{\ell} \right)^{3/2}(4k v_3 - 1.3536)\leq
\vol(S^3 \setminus K) < 4v_8k,$$
where $v_3 = 1.0149...$ is the volume of a regular ideal tetrahedron
and $v_8 = 3.6638...$ is the volume of a regular ideal octahedron in
$\HH^3$.
\label{thm:fkp-volbound}
\end{theorem}

\begin{proof}
By Definition \ref{def:aug-link}, $K$ is
obtained by $1/n_i$ filling on the component $C_i$
of $L_{p/q}$. Thus the upper bound on volume follows immediately from Theorem
\ref{thm:parent} and the fact that the volume decreases under Dehn
filling \cite[Theorem 6.5.6]{thurston:notes}.

For the lower bound, let $\lmin$ denote the minimum of the lengths of
$1/n_1$ and $1/n_2$ in some horoball expansion about the cusps
corresponding to $C_1$ and $C_2$.  Provided $\lmin > 2\pi$, Theorem
\ref{thm:fkp-dehn-filling} implies that $K$ is hyperbolic and:
\begin{eqnarray*}
\vol(S^3 \setminus K) \:  & \geq \: &
\left( \! 1 - \left(\frac{2\pi}{\lmin} \right)^{\! 2} \right)^{\! 3/2}
 \! \vol(S^3\setminus L_{p/q}) \\
 & \geq \: &
\left( \! 1 - \left( \frac{2\pi}{\lmin} \right)^{\! 2} \right)^{\!
	3/2}
 \! (4 k v_3 - 1.3536).
\end{eqnarray*}

Thus we determine some admissible values of $n_1$, $n_2$, and $k$, for which the slopes
$1/n_1$ and $1/n_2$ are both guaranteed to have length at least $2\pi$
under some horoball expansion.  First, recall that by Lemma
\ref{lemma:aug-symmetries}(a) we may arrange the diagram of $L_{p/q}$
such that the link $L_{p/q}$ is fixed under reflection in the
projection plane.  It follows immediately from \cite[Proposition
3.5]{purcell:aug-ref-tw} that there exists a horoball expansion about
these cusps such that the slope $1/n_i$ has length at least
\begin{equation}\label{eq:simple-lengh-est}
\ell_i \: \geq \:  \sqrt{1/4 + 4n_i^2}.
\end{equation}
This quantity is greater than $2\pi$ when $n_i \geq 4$.  Hence the
conclusion follows in this case.

On the other hand, when $k$ is relatively large, we can get a better
estimate on the lengths of the slopes $1/n_1$ and $1/n_2$.  In
\cite[Theorem 4.8]{fkp:farey}, we found bounds on the length of
certain arcs on the cusps of 2--bridge knots.  In particular, the
shortest non-trivial arc running from a meridian back to that meridian
in a 2--bridge knot with $(k+1)$ twist regions has length at least
$(4\sqrt{6\sqrt 2}) \, k / 147$.  In other words, the area of the
maximal cusp about the knot is at least $(4\sqrt{6\sqrt 2}) \, k \cdot
\mu / 147$, where $\mu$ is the length of the meridian.  Since a
clasped $2$--bridge link of slope $p/q$ is the geometric limit of
$2$--bridge knots with $(k+1)$ twist regions, the same estimate
applies to the clasped link.

Now, by Proposition \ref{thm:2bridge}, $L_{p/q}$ is obtained as the
belted sum of clasped $2$--bridge links of slope $p/q$ and $-p/q$.
Consider what happens to the cusps during the gluing process.  The
cusps about $C_1$ and $C_2$ come from the knot component(s) of the
clasped link, i.e. the component(s) which form the 2--bridge link
rather than the clasp.  The meridians of $C_1$ and $C_2$ agree with
meridians of the 2--bridge link, and both have length $\mu$.
Furthermore, the total area of the cusps about $C_1$ and $C_2$ is
equal to twice the area of the cusp about the $2$--bridge knot, namely
at least $(8\sqrt{6\sqrt 2})\,k \cdot \mu /147$.  But by Lemma
\ref{lemma:aug-symmetries}(b), there is a symmetry of $L_{p/q}$
interchanging $C_1$ and $C_2$, hence each of those cusps has area at
least $(4\sqrt{6\sqrt 2}) \, k \cdot \mu /147$.  As a result, in each
of $C_1$ and $C_2$, the shortest non-trivial arc running from a
meridian back to that meridian has length at least
$(4\sqrt{6\sqrt 2})\, k /147$.

Finally, note that the slope $1/n_i$ crosses the meridian exactly
$\abs{n_i}$ times.  Since each non-trivial arc from the meridian to
the meridian has length bounded as above, the total length of the
slope is at least $(4\sqrt{6\sqrt 2}) \, k \abs{n_i} / 147$.  When
$k \abs{n_i} \geq 80$, the slope $1/n_i$ will be longer than $2\pi$,
and the desired volume estimate follows.
\end{proof}


Theorem \ref{thm:simple-volume}, stated in the introduction, is now
an immediate corollary of Theorem \ref{thm:fkp-volbound}.

\begin{theorem}\label{thm:simple-volume}
Let $p$ and $q$ be relatively prime integers with $0<p<q$, and let $k$ be the length of the continued fraction expansion of $p/q$. Let $K$ be a $(p,q)$ double coil knot, in which each generalized twist region has at least $4$ full twists. Then $K$ is hyperbolic, and
$$0.9718 \, k - 0.3241 \: \leq \:  \vol(S^3 \setminus K) \: < \: 4v_8k,$$
\end{theorem}

\begin{proof}
If $\abs{n_i} \geq 4$ for $i = 1, 2$, then $\ell \geq 64.25$ in
equation (\ref{eq:ell}).  Plugging this estimate into Theorem
\ref{thm:fkp-volbound}, and substituting the numerical values of all
the constants in the lower bound on volume, gives the desired result.
\end{proof}

\section{Volume and twist number}\label{sec:volume-twist}

The twist number of a knot $K$, denoted by $\tau(K)$, is defined to be
the minimum twist number over all knot diagrams of $K$; it is clearly
an invariant of $K$.  In this section, we describe a general
construction of hyperbolic knots with bounded volume and arbitrarily
large twist number.

\begin{theorem}\label{thm:big-twisting}
Let $K$ be a knot in $S^3$, and let $U \subset S^3\setminus K$ be an
unknot in $S^3$ with the property that $S^3 \setminus (K \cup U)$ is
hyperbolic.  Suppose that every disk bounded by $U$ intersects $K$ at
least three times.  Let $K_n$ be the knot obtained by $1/n$ surgery on
$U$.  Then, for $\abs{n}$ sufficiently large, $K_n$ is a hyperbolic
knot, with
$$\vol(S^3 \setminus  K_n) < \vol(S^3 \setminus (K \cup U))
\quad \mbox{and} \quad \lim_{\abs{n}\to \infty} \tau(K_n)=\infty.
$$
\end{theorem}

\begin{proof}
By Thurston's hyperbolic Dehn surgery theorem \cite{thurston:notes},
$K_n$ is hyperbolic for $\abs{n}$ large enough.  Furthermore, because
volume decreases strictly under Dehn filling \cite[Theorem
6.5.6]{thurston:notes}, we have $\vol(S^3 \setminus K_n) < \vol(S^3
\setminus (K\cup U))$.  Now the complement $S^3 \setminus (K\cup U)$
is the geometric limit of $S^3 \setminus K_n$, as $\abs{n}\to \infty$.
By the proof of the hyperbolic Dehn surgery theorem, for $n$ large
enough, the core of the Dehn filling torus (that is, $U$) is the
unique minimum--length geodesic in $S^3 \setminus K_n$, and the length
of this geodesic goes to $0$ as $n\to \infty$.  Note that, since $K_n$
is obtained from $K$ by twisting about the disk bounded by $U$, this
disk intersects $K_n$ the same number of times as $K$, namely at least
three.

Now we argue that $\lim_{\abs{n}\to \infty} \tau(K_n)=\infty$.
Suppose not: assume that $\tau(K_n)$ is bounded independently of $n$.
Then consider the knot $K_n$.
Take a diagram of $K_n$ for which the twist number is minimal, equal
to $\tau(K_n)$.  Encircle each twist region of the diagram by a
crossing circle.  The result is a \emph{fully augmented link}, and the
knot $K_n$ is obtained by Dehn filling this fully augmented link. 
(For an example with two twist regions, compare Figure
\ref{fig:dfr-knot}, left to Figure \ref{fig:dfr-link}, left.)

To obtain the standard diagram of a fully augmented link, we remove
all pairs of crossings (full twists) from each crossing circle.  (See
also, for example, Figure 6 of \cite{fkp:volume}.)  Thus the standard
diagram of a fully augmented link with $\tau$ twist regions consists
of $\tau$ crossing circles encircling two strands each, possibly with
a single crossing at each crossing circle.  This can be represented by
a 4-valent graph with $\tau$ vertices, and a choice of crossing at
each vertex.  Since $\tau(K_n)$ is bounded independently of $n$, there
are only finitely many such 4-valent graphs, so only finitely many
fully augmented links.  As a result, there must be an infinite
subsequence of knots $K_n$ that is obtained by surgery on a single
augmented link.

Recall that this subsequence converges geometrically to $S^3 \setminus
(K\cup U)$.  Thus, in each twist region in this infinite subsequence,
the number of crossings either becomes eventually constant or goes to
infinity.  Since infinitely many of the $K_n$ are distinct knots,
there is at least one twist region whose number of crossings goes to
infinity.  Furthermore, since the geometric limit of $S^3 \setminus
K_n$ is the manifold $S^3 \setminus (K\cup U)$ with exactly two cusps,
there must be exactly one twist region whose number of crossings goes
to infinity, for if the number of crossings goes to infinity, the
geometric limit yields an additional cusp \cite{thurston:notes}.  

Hence we have an infinite subsequence of the knots $K_n$ that is
obtained from a 2--component link $K' \cup U'$ by Dehn filling along
an unknotted component $U'$.  Furthermore, the subsequence $S^3
\setminus K_n$ converges geometrically to $S^3 \setminus (K' \cup
U')$, and the crossing circle $U'$ bounds a disc whose interior is
pierced exactly twice by $K'$.  Furthermore, under this convergence the
core $U'$ eventually becomes the unique minimal geodesic in
$S^3\setminus K_n$.

Now recall that for $\abs{n}$ sufficiently large, $U$ is also the
unique minimal--length geodesic in $S^3 \setminus K_n$ .  We conclude
that for $\abs{n}$ large enough there must be an isometry $S^3
\setminus K_n \longrightarrow S^3 \setminus K_n$ that maps $U$ onto
$U'$.  However, $U'$ bounds a disk whose interior is punctured twice
by $K_n$, whereas $U$ does not.  This is a contradiction.
\end{proof}

One way to construct a sequence of knots satisfying Theorem
\ref{thm:big-twisting} is the following.

\begin{corollary}\label{cor:braids}
Fix an integer $m \geq 3$. Then there is a sequence $K_n$ of
hyperbolic closed $m$--braids, such that $\vol(S^3 \setminus K_n)$ is
bounded but $\tau(K_n)$ is unbounded.
\end{corollary}

\begin{proof}
Given a closed $m$--braid $K$, let $A$ be the braid axis of $K$.  That
is: $S^3 \setminus A$ is a solid torus swept out by meridian disks,
with each disk intersecting $K$ in $m$ points.  The complement $S^3
\setminus (K \cup A)$ is a fiber bundle over $S^1$, with fiber an
$m$--punctured disk.  By a theorem of Thurston
\cite{thurston:fibered-manifolds}, this manifold will be hyperbolic
whenever the monodromy is pseudo--Anosov.  Furthermore, since the
fiber minimizes the Thurston norm within its homology class, the
unknot $A$ does not bound a disk meeting $K$ in fewer than $m$ points.
Thus Theorem \ref{thm:big-twisting} applies, and the sequence of knots
$K_n$ obtained by $1/n$ filling on $A$ has bounded volume but
unbounded twist number.
\end{proof}

For the \dfr knots studied in the last section, Theorem \ref{thm:big-twisting}
applies to give:

\begin{theorem}\label{thm:no-easy-estimate}
The volumes of hyperbolic \dfr knots are not effectively predicted by
either the twist number or generalized twist number.  More precisely:
\begin{enumerate}
\item[(a)] For any $q \geq 3$, and any $p$ relatively prime to $q$,
there exists a sequence $K_n$ of $(p,q)$ \dfr knots such that
$\tau(K_n) \to \infty$ while $\vol(S^3 \setminus K_n)$ stays bounded.
\item[(b)] All double coils have generalized twist number $2$, but their
volumes are unbounded.
\end{enumerate}
\end{theorem}

\begin{proof}
Statement (b) is an immediate consequence of Definition
\ref{def:dfr-knot} and Theorem \ref{thm:simple-volume}.

For statement (a), consider the sequence $K_n$ of \dfr knots obtained
from $L_{p/q}$ by $1/n$ filling on the circle $C_1$ and $1/6$ filling
on $C_2$. When $n \geq 4$, Theorem \ref{thm:simple-volume} implies
that each $K_n$ is hyperbolic. The volumes of $S^3 \setminus K_n$ are
bounded above by the volume of $S^3 \setminus L_{p/q}$.  To apply
Theorem \ref{thm:big-twisting} and show that the twist number of $K_n$
is unbounded, we need to show that every disk bounded by $C_1$ meets
$K_n$ at least three times.

Suppose, for a contradiction, that $D$ is a disk in $S^3$ whose
boundary is $C_1$, and such that $\abs{K_n \cap D} \leq 2$.  Since
$S^3 \setminus K_n$ is hyperbolic, it cannot contain any essential
disks or annuli. Thus $K_n$ meets $D$ exactly twice. We assume that
$D$ has been moved by isotopy into a position that minimizes its
intersection number with $C_2$, and consider two cases.

\smallskip

\emph{\underline{Case 1: $D$ is disjoint from $C_2$.}}  Then when
$C_1$ and $C_2$ are drilled out of $S^3 \setminus K_n$, $D$ becomes a
disk in $S^3\setminus L_{p/q}$ that intersects $K$ twice, where $K$ is
the planar curve of slope $p/q$ that will become $K_n$ after Dehn
filling on $C_1$ and $C_2$.  Consider the standard diagram of
$L_{p/q}$, with all full--twists removed from generalized twist
regions.  Isotope $D$ so that it intersects the projection plane of
this diagram transversely a minimal number of times.
Then the intersection between $D$ and the $4$--punctured projection
sphere $S$ consists of some number of simple closed curves, as well as
exactly one arc $\alpha$ connecting two of the punctures of $S$.
(These punctures are the intersections between $C_1$ and the
projection plane --- see Figure \ref{fig:gens}.)

Since $D$ intersects $K$ twice, and
$K$ lies in the projection plane as a curve of slope $p/q$, this arc
$\alpha \subset D$ must intersect the curve of slope $p/q$ at most twice.  On
the other hand, since $\alpha$ lies in a disk whose boundary is $C_1$,
$\alpha$ must be isotopic to one half of $C_1$, in other words to an
arc of slope $1/0$ in $S$.  But it is well--known (for example, by
passing to the universal abelian cover $\RR^2 \setminus \ZZ^2$) that
in a $4$--punctured sphere, the arc of slope $1/0$ and the closed
curve of slope $p/q$ must intersect at least $q$ times. Since $q \geq
3$, this is a contradiction.

\smallskip

\emph{\underline{Case 2: $D$ is not disjoint from $C_2$.}} Let $E = D
\cap (S^3 \setminus L_{p/q})$.  Then $E$ is a sphere with $(r+3)$
holes, where one boundary circle is at the cusp of $C_1$, two boundary
components are at the cusp of $K$, and $r$ boundary components are at
the cusp of $C_2$. Consider the length $\ell$ of the Dehn filling
slope along $C_2$, where $r \geq 1$ boundary circles of $E$ run in
parallel along the cusp.  A result of Agol and Lackenby (see
\cite[Theorem 5.1]{agol:6theorem} or \cite[Lemma
3.3]{lackenby:surgery}) implies that the total length of those circles
is
$$r \, \ell \: \leq \: -6 \, \chi(E) \: = \: 6(r+1) \: \leq \: 12 r.$$
Thus $\ell \leq 12$. On the other hand, since we are filling $C_2$
along slope $1/6$, equation (\ref{eq:simple-lengh-est}) above implies
that
$$\ell \: \geq \: \sqrt{1/4 + 4 \cdot 6^2} \: = \: \sqrt{144.25} \: >
\: 12.$$ Therefore, in this case as well as in Case 1, we obtain a
contradiction.
\end{proof}


\section{Spectral geometry}\label{sec:laplacian}

In this section, we investigate the spectral geometry of \dfr knot
complements. Recall that, for a Riemannian manifold $M$, $\lam(M)$ is
defined to be the smallest positive eigenvalue of the
Laplace--Beltrami operator $\Delta f = - {\rm div \, grad} f$. When
$M$ is a hyperbolic $3$--manifold, it is known that $\lam(M)$ has many
connections to the volume of $M$. The following result is essentially
a combination of theorems by Schoen \cite{schoen:spectrum}, Dodziuk
and Randol \cite{dodziuk-randol}, Lackenby \cite{lackenby:gradient},
and Buser \cite{buser}.

\begin{theorem}\label{thm:lambda-general}
Let $M$ be an oriented, finite--volume hyperbolic $3$--manifold. Then
$$\frac{\pi^2 / 2^{50}}{\vol(M)^2} \: \leq \: \lam(M) \: \leq \: 32
\pi \, \frac{g(M)-1}{\vol(M)} + 640 \pi^2 \,
\frac{(g(M)-1)^2}{\vol(M)^2},$$ where $g(M)$ is the Heegaard genus of
$M$.
\end{theorem}

To write down the proof of Theorem \ref{thm:lambda-general}, we need
the following fact.

\begin{lemma}\label{lemma:vol-stupid}
An oriented, finite--volume hyperbolic $3$--manifold $M$ satisfies
$\vol(M) > \pi/2^{25}$.
\end{lemma}

\begin{proof}
Gabai, Meyerhoff, and Milley recently showed
\cite{gmm:smallest-cusped} that the unique lowest--volume orientable
hyperbolic $3$--manifold is the Weeks manifold of volume $ \approx
0.9427$. This is the culmination of many increasingly sharp estimates,
by a number of hyperbolic geometers. In fact, Meyerhoff's 1984 result
\cite{meyerhoff:first-estimate} that that $\vol(M) \geq 0.00064$ is
several orders of magnitude larger than necessary for this lemma.
\end{proof}

\begin{proof}[Proof of Theorem \ref{thm:lambda-general}]
Dodziuk and Randol \cite{dodziuk-randol} showed that for all
finite--volume, hyperbolic $n$--manifolds (where $n \geq 3$), $\lam(M)
\geq A(n) /\vol(M)^2$, where the constant $A(n)$ depends only on the
dimension $n$. To estimate $A(3)$ for dimension $3$, we rely on the
work of Schoen, who gave an explicit estimate for $\lam(M)$ when $M$
is closed and negatively curved \cite{schoen:spectrum}. In the special
case where $M$ is a closed, hyperbolic $3$--manifold, his theorem says
that
$$
 \lam(M) \: \geq \: \min \left\{ 1, \;\; \frac{\pi^2}{2^{50}} \cdot
\frac{1}{\vol(M)^2} \right\} 
\: \geq \: \frac{\pi^2 / 2^{50}}{\vol(M)^2},
$$
where the second inequality is Lemma \ref{lemma:vol-stupid}.  This
completes the proof of the lower bound on $\lam(M)$ in the case where
$M$ is closed.

Now, suppose that $M$ has cusps. We may assume that $\lam(M) < 1$;
otherwise, $\lam(M)$ already satisfies the desired lower bound by
Lemma \ref{lemma:vol-stupid}.  Let $N_i$ be a sequence of closed
manifolds obtained by Dehn filling $M$, along slopes whose lengths
tend to infinity.  Thurston's Dehn surgery theorem
\cite{thurston:notes} implies that the manifolds $N_i$ approach $M$ in
the geometric topology; in particular, $\vol(N_i) \to \vol(M)$.
Meanwhile, assuming that $\lam(M) < 1$, Colbois and Courtois
\cite[Theorem 3.1]{colbois-courtois} showed that $\lam(N_i) \to
\lam(M)$. Thus, since the lower bound on $\lam$ holds for each closed
$N_i$, it also holds for $M$.

The upper bound on $\lam(M)$ is a combination of results by Buser
\cite{buser} and Lackenby \cite{lackenby:gradient}.  Buser proved an
inequality relating $\lam(M)$ to the \emph{Cheeger constant} $h(M)$,
defined by
$$h(M) := \inf\left\{ \frac{\area(S) }{\min(V_1, V_2) } \right\},$$
where $S$ is a separating surface in $M$, and $V_1$, $V_2$ are the
volumes of the two pieces separated by $S$.  Buser's result
\cite{buser} says that
$$\lam(M) \leq 4h(M) + 10h(M)^2.$$
More recently, Lackenby showed \cite[Theorem 4.1]{lackenby:gradient}
that if a hyperbolic manifold $M$ has a genus--$g$ Heegaard splitting,
$$h(M) \leq \frac{8 \pi (g-1)}{\vol(M)}. $$ Plugging this estimate
into Buser's inequality yields the upper bound on $\lam(M)$.
\end{proof}

For \dfr knots, Theorem \ref{thm:lambda-general} implies the following
result.

\begin{theorem}\label{thm:dfr-lambda-bound}
Let $K$ be a hyperbolic \dfr knot. Then
$$ \frac{A_1}{\vol(S^3 \setminus K)^2} \: \leq \: \lam(S^3 \setminus
K) \: \leq \: \frac{A_2}{\vol(S^3 \setminus K)},$$
where $A_1 \geq  8.76 \times 10^{-15}$ and $A_2 \leq 12650$.
\end{theorem}

\begin{proof}
The lower bound on $\lam$ is a restatement of Theorem
\ref{thm:lambda-general}.  Note $\pi^2 / 2^{50} \approx 8.765 \times
10^{-15}.$

To establish the upper bound on $\lam$, we bound the Heegaard genus of
$S^3 \setminus K$.  Recall that $K$ is obtained by Dehn filling two
components of the link $L_{p/q}$ depicted in Figure \ref{fig:scheme}.
Since each of the boxes in Figure \ref{fig:scheme} contains a braid,
the figure is a $3$--bridge diagram of $L_{p/q}$.  It is well--known
that a $g$--bridge link $L$ has Heegaard genus at most $g$.  (One
standard way to obtain a Heegaard surface is to connect the maxima in
a $g$--bridge diagram of $L$ by $g-1$ arcs, thicken the union of $L$
and these arcs, and take the boundary of the resulting genus--$g$
handlebody. The exterior of this handlebody is unknotted, because $L$
was in bridge position. See \cite[Figure 1]{sakuma:survey}.)  Thus
$S^3 \setminus L_{p/q}$ has Heegaard genus at most $3$.  Since
Heegaard genus can only go down under Dehn filling, $S^3 \setminus K$
also has Heegaard genus at most $3$.

Plugging $g(S^3 \setminus K) \leq 3$ into Theorem
\ref{thm:lambda-general}, we obtain {\setlength{\jot}{1.3ex}
\begin{eqnarray*}
\lam(S^3 \setminus K) & \leq &
 \frac{64 \pi}{\vol(S^3 \setminus K)} \, + \,
 \frac{2560 \pi^2}{\vol(S^3 \setminus K)^2}  \\
&\leq&  \frac{64 \pi}{\vol(S^3 \setminus K)} \, + \,
 \frac{2560 \pi^2}{\vol(S^3 \setminus K) \cdot 2v_3}   \\
&<& \frac{12650}{\vol(S^3 \setminus K)},
\end{eqnarray*}
}
where the second inequality follows because the smallest--volume
knot is the figure--8 knot, with $\vol(S^3 \setminus K) = 2v_3$
\cite{cao-meyerhoff}.
\end{proof}

A collection $\{ M_i\}$ of hyperbolic $3$--manifolds is called an
\emph{expanding family} if $ \inf \{\lam(M_i)\}>0$, that is,
$\lam(M_i)$ is bounded away from $0$. With this notation, Theorem
\ref{thm:dfr-lambda-bound} has the following immediate corollary.

\begin{corollary}\label{cor:dfr-expanding}
Let $\{ K_i\}$ be a collection of hyperbolic \dfr knots.  Then
$\{\lam(S^3 \setminus K_i)\}$ is bounded away from $0$ if and only if
$ \{\vol(S^3 \setminus K_i)\}$ is bounded above.  In other words, the
knots $\{ K_i\}$ form an expanding family if and only if their volumes
are bounded.
\end{corollary}

Corollary \ref{cor:dfr-expanding} is significant in light of recent
work of Lackenby \cite{lackenby:spectrum}.  He showed that for two
large families of hyperbolic links (namely, alternating links and
highly twisted links), $\lam(S^3 \setminus K)$ is bounded above in
terms of the inverse of the twist number of a sufficiently reduced
diagram.  Because the volumes of these links are also governed by the
twist number \cite{fkp:volume, lackenby:volume}, it follows that
alternating and highly twisted links form an expanding family if and
only if their volumes are bounded \cite[Corollary
1.7]{lackenby:spectrum}.  Corollary \ref{cor:dfr-expanding} is the
analogous result for \dfr knots.

On the other hand, by Theorem \ref{thm:no-easy-estimate}, the volumes
of \dfr knots are not governed by the twist number in any meaningful
sense. Thus, combining Theorem \ref{thm:no-easy-estimate} with
Corollary \ref{cor:dfr-expanding} yields the following result.

\begin{corollary}\label{cor:lambda-twist}
The spectrum of the Laplacian of hyperbolic \dfr knots is not
effectively predicted by either the twist number or generalized twist
number. More precisely:
\begin{enumerate}
\item[(a)] For any $q \geq 3$, and any $p$ relatively prime to $q$,
there exists a sequence $K_n$ of $(p,q)$ double coils such that
$\tau(K_n) \to \infty$ while $\lam(S^3 \setminus K_n)$ is bounded away
from $0$ and $\infty$.
\item[(b)] All \dfr knots have generalized twist number $2$, but the
	infimum of $\{\lam(S^3 \setminus K_n)\}$ is zero.
\end{enumerate}
\end{corollary}

\begin{proof}
Each part of this corollary follows by combining Theorem
\ref{thm:dfr-lambda-bound} with the corresponding part of Theorem
\ref{thm:no-easy-estimate}.
\end{proof}

\bibliographystyle{hamsplain} \bibliography{biblio.bib}

\end{document}